\documentclass{amsart}
\usepackage{amssymb}
\usepackage{amsfonts}
\usepackage{hyperref}

\newtheorem{theorem}{Theorem}
\theoremstyle{plain}

\newtheorem{conjecture}{Conjecture}

\newtheorem{lemma}{Lemma}

\newtheorem{remark}{Remark}

\numberwithin{equation}{section}
\input{tcilatex}

\begin{document}
\title[Monotonicity and Sharp inequalities for type-power means]{The
monotonicity results and sharp inequalities for some power-type means of two
arguments}
\address{Power Supply Service Center, ZPEPC Electric Power Research
Institute, Hangzhou, Zhejiang, China, 31001}
\email{yzhkm@163.com}
\date{October 16, 2012}
\thanks{This paper is in final form and no version of it will be submitted
for publication elsewhere.}

\begin{abstract}
For $a,b>0$ with $a\neq b$, we define 
\begin{equation*}
M_{p}=M^{1/p}(a^{p},b^{p})\text{ if }p\neq 0\text{ and }M_{0}=\sqrt{ab},
\end{equation*}%
where $M=A,He,L,I,P,T,N,Z$ and $Y$ stand for the arithmetic mean, Heronian
mean, logarithmic mean, identric\ (exponential) mean, the first Seiffert
mean, the second Seiffert mean, Neuman-S\'{a}ndor mean, power-exponential
mean and exponential-geometric mean, respectively. Generally, if $M$ is a
mean of $a$ and $b$, then $M_{p}$ is also, and call "power-type mean". We
prove the power-type means $P_{p}$, $T_{p}$, $N_{p}$, $Z_{p}$ are increasing
in $p$ on $\mathbb{R}$ and establish sharp inequalities among power-type
means $A_{p}$, $He_{p}$, $L_{p}$, $I_{p}$, $P_{p}$, $N_{p}$, $Z_{p}$, $Y_{p}$%
. From this a very nice chain of inequalities for these means 
\begin{equation*}
L_{2}<P<N_{1/2}<He<A_{2/3}<I<Z_{1/3}<Y_{1/2}.
\end{equation*}%
follows. Lastly, a conjecture is proposed.
\end{abstract}

\author{Zhen-Hang Yang}
\maketitle

\section{Introduction}

There are many basic bivariate means of positive numbers $a$ and $b$, such as

\begin{itemize}
\item the arithmetic mean $A$ defined by 
\begin{equation}
A\left( a,b\right) =\frac{a+b}{2};  \label{A}
\end{equation}

\item geometric mean $G$ defined as \bigskip 
\begin{equation}
G\left( a,b\right) =\sqrt{ab};  \label{G}
\end{equation}

\item Heronian mean $He$ defined by 
\begin{equation}
He\left( a,b\right) =\frac{a+b+\sqrt{ab}}{3};  \label{He}
\end{equation}

\item logarithmic mean $L$ defined by%
\begin{equation}
L\left( a,b\right) =\frac{a-b}{\ln a-\ln b}\text{ if }a\neq b\text{ and }%
L\left( a,a\right) =a;  \label{L}
\end{equation}

\item identric (exponential) mean defined by 
\begin{equation}
I\left( a,b\right) =e^{-1}\left( \frac{a^{a}}{b^{b}}\right) ^{1/(a-b)}\text{
if }a\neq b\text{ and }I\left( a,a\right) =a;  \label{I}
\end{equation}

\item the first Seiffert mean $P$, defined in \cite{Seiffert.42(1987)} as 
\begin{equation}
P\left( a,b\right) =\frac{a-b}{2\arcsin \frac{a-b}{a+b}}\text{ if }a\neq b%
\text{ and }P\left( a,a\right) =a;  \label{P}
\end{equation}

\item the second Seiffert mean $T$, defined in \cite{Seiffert.29(1995)} by 
\begin{equation}
T\left( a,b\right) =\frac{a-b}{2\arctan \frac{a-b}{a+b}}\text{ if }a\neq b%
\text{ and }T\left( a,a\right) =a;  \label{T}
\end{equation}

\item Neuman-S\'{a}ndor mean $N$, defined in \cite{Neuman.17(1)(2006)} by 
\begin{equation}
N\left( a,b\right) =\frac{a-b}{2\func{arcsinh}\frac{a-b}{a+b}}\text{ if }%
a\neq b\text{ and }N\left( a,a\right) =a;  \label{N}
\end{equation}

\item power-exponential mean, the special case of Gini means \cite%
{Gini.13.1938}, defined by 
\begin{equation}
Z\left( a,b\right) =a^{\frac{a}{a+b}}b^{\frac{b}{a+b}};  \label{Z}
\end{equation}

\item exponential-geometric mean, defined in \cite{Yang.6(4).101.2005} by 
\begin{equation}
Y\left( a,b\right) =I\exp \left( 1-\frac{G^{2}}{L^{2}}\right) ,  \label{Y}
\end{equation}%
where $I,G,L$ denote the identric mean, geometric mean and logarithmic mean
of positive numbers $a$ and $b$.
\end{itemize}

We define%
\begin{equation}
M_{p}:=M_{p}(a,b)=M^{1/p}(a^{p},b^{p})\text{ if }p\neq 0\text{ and }M_{0}=%
\sqrt{ab},  \label{M_p...}
\end{equation}%
where $M=A,He,L,I,P,T,N,Z$ and $Y$ stand for the arithmetic mean, Heroian
mean, logarithmic mean, identric (exponential) mean, the first Seiffert
mean, the second Seiffert mean, Neuman-S\'{a}ndor mean, power-exponential
mean and exponential-geometric mean, which are defined by (\ref{A})- (\ref{Y}%
), respectively. It is known that $A_{p}$ is the classical power mean which
is increasing with $p$ on $\mathbb{R}$.

Also, we note that for $t\in \mathbb{R}$ 
\begin{equation}
M_{pt}^{t}\left( a,b\right) =M^{1/p}\left( a^{pt},b^{pt}\right) =M_{p}\left(
a^{t},b^{t}\right) .  \label{M_pt}
\end{equation}

In the most cases, ones more prefer to evaluate a given more complicated
mean $M$ by a simpler one such as arithmetic mean, geometric mean, Heroian
mean and power mean etc. For example, for $a,b>0$ with $a\neq b$, Lin \cite%
{Lin.81.1974} gave a best estimation of the logarithmic mean $L$ by power
means, that is, 
\begin{equation}
A_{0}\left( a,b\right) <L\left( a,b\right) <A_{1/3}\left( a,b\right) .
\label{L-A_p}
\end{equation}%
Jiao and Cau proved in \cite{Jiagao.4(4).2003}\ that 
\begin{equation}
L\left( a,b\right) <He_{1/2}\left( a,b\right) <A_{1/3}\left( a,b\right) .
\label{L-He_p}
\end{equation}%
Stolarsky \cite{Stolarsky.87.1980} and Pittenger \cite{Pittinger.680(1980)}
showed that the inequalities 
\begin{equation}
A_{2/3}\left( a,b\right) <I\left( a,b\right) <A_{\ln 2}\left( a,b\right)
\label{I-A_p}
\end{equation}%
hold, where the constants $2/3$ and $\ln 2$ are the best possible. The
following sharp double inequality 
\begin{equation}
A_{\ln 2/\ln 3}\left( a,b\right) <He\left( a,b\right) <A_{2/3}\left(
a,b\right)  \label{He-A_p}
\end{equation}%
is due to Alzer and Janous \cite{Alzer.13(173-178)(1987)}.

For the first Seiffert mean, Jagers first established in \cite%
{Jagers.12(1994)} (also see \cite{Hasto.3(5)(2002)}) that 
\begin{equation*}
A_{1/2}\left( a,b\right) <P\left( a,b\right) <A_{2/3}\left( a,b\right) ,
\end{equation*}%
which has been improved by H\"{a}st\"{o} in \cite{Hasto.7(1)(2004)} as 
\begin{equation}
A_{\log _{\pi }2}\left( a,b\right) <P\left( a,b\right) <A_{2/3}\left(
a,b\right) ,  \label{P-A_p}
\end{equation}%
where $\log _{\pi }2$ and $2/3$ are the best possible. In 1995, Seiffert 
\cite{Seiffert.29(1995)} indicated that 
\begin{equation}
A\left( a,b\right) <T\left( a,b\right) <A_{2}\left( a,b\right) ,
\label{Seiffert}
\end{equation}%
which was refined by Yang \cite{Yang.arXiv:1206.5494} as%
\begin{equation}
A_{\log _{\pi /2}2}\left( a,b\right) <T\left( a,b\right) <A_{5/3}\left(
a,b\right) ,  \label{T-A_p}
\end{equation}%
where $\log _{\pi /2}2$ and $3/5$ can not be improved. Utilizing (\ref{M_pt}%
) the second one of \ref{T-A_p} can be written as 
\begin{equation}
T_{2/5}\left( a,b\right) <A_{2/3}\left( a,b\right) .  \label{T_p-A_q}
\end{equation}%
Chu et al. showed in \cite{Chu.146945.2010} an optimal double inequality 
\begin{equation}
He_{\frac{\ln 3}{\ln \left( \pi /2\right) }}<T\left( a,b\right)
<He_{5/2}\left( a,b\right) ,  \label{T-He_p}
\end{equation}%
the second one in which is equivalent to 
\begin{equation}
T_{2/5}\left( a,b\right) <He\left( a,b\right) .  \label{T_P-He}
\end{equation}

For the Neuman-S\'{a}ndor mean, Yang \cite{Yang.arXiv:1208.0895} has
presented sharp bounds in terms of power means, that is,

\begin{equation}
\allowbreak A_{\frac{\ln 2}{\ln \ln \left( 3+2\sqrt{2}\right) }}\left(
a,b\right) <N\left( a,b\right) <A_{4/3}\left( a,b\right) ,  \label{N-A_p}
\end{equation}%
which by using (\ref{M_pt}) implies that 
\begin{equation}
N_{1/2}\left( a,b\right) <A_{2/3}\left( a,b\right) .  \label{N_p-A_q}
\end{equation}

For the power-exponential mean $Z$, from the comparison theorem for Gini
means given by P\'{a}les in \cite{Pales.131(1988)} (also see \cite%
{Acu.9(1-2)(2001)}, \cite{Sandor.12(4)(2004)}, \cite{Yang.10(3).2007}) it is
easy to obtain the following optimal inequality:%
\begin{equation}
Z\left( a,b\right) >A_{2}\left( a,b\right) ,  \label{Z-A_p}
\end{equation}%
which also can be written as 
\begin{equation}
Z_{1/3}\left( a,b\right) >A_{2/3}\left( a,b\right) ,  \label{Z_p-A_q}
\end{equation}

On the other hand, S\'{a}ndor showed in \cite{Sandor.76(2001)} that%
\begin{equation}
L<P<I.  \label{L-P-I}
\end{equation}%
Neuman and S\'{a}ndor \cite{Neuman.17(1)(2006)}\ established the following
chain of inequalities for means:%
\begin{equation*}
G<L<P<A<N<T<A_{2}.
\end{equation*}%
The following chain of inequalities for means%
\begin{equation}
L_{2}<He<A_{2/3}<I<Z_{1/3}<Y_{1/2}  \label{Yang1}
\end{equation}%
is due to Yang \cite[(5.17)]{Yang.10(3).2007}. Recently, Costin and Toader 
\cite{Costin.IJMMS.2012.inprint} presented a nice separation of some
Seiffert type means by power means:%
\begin{equation}
G<L<A_{1/2}<P<A<N<T<A_{2},  \label{Costin-Toader}
\end{equation}%
which has been improved by Yang \cite{Yang.arXiv:1208.0895} as 
\begin{eqnarray}
A_{0} &<&L<A_{1/3}<A_{\ln _{\pi }2}<P<A_{2/3}<I<A_{\ln 2}  \label{Yang2} \\
&<&A_{\frac{\ln 2}{\ln \ln \left( 3+2\sqrt{2}\right) }}<N<A_{4/3}<A_{\log
_{\pi /2}2}<T<A_{5/3}.  \notag
\end{eqnarray}

Motived by these inequalities for bivariate means, the purpose of this paper
is to investigate the monotonicities of $P_{p},T_{p},N_{p},Z_{p}$ in $p$ (in
the next section, we will prove that $M_{p}$ is also a mean and call it
"power-type mean") and establish the relations among $M_{p}$ defined by (\ref%
{M_p...}).

\section{The monotonicities of power-type means}

In general, a function $M:\mathbb{R}_{+}^{2}\mapsto \mathbb{R}$ is called a
bivariate mean if 
\begin{equation*}
\min \left( a,b\right) \leq M\left( a,b\right) \leq \max \left( a,b\right)
\end{equation*}%
holds for all $a,b>0$. Clearly, each bivariate mean is reflexive, that is, 
\begin{equation*}
M\left( a,a\right) =a\text{ for all }a>0.
\end{equation*}%
A bivariate mean is symmetric if 
\begin{equation*}
M\left( a,b\right) =M\left( b,a\right)
\end{equation*}%
holds for all $a,b>0$. It is said to be homogeneous (of degree one) if 
\begin{equation*}
M\left( ta,tb\right) =tM\left( a,b\right)
\end{equation*}%
holds for all $a,b,t>0$.

Let $M$ be a differentiable mean on $\mathbb{R}_{+}^{2}$. Now we introduce
the function $M_{p}:\mathbb{R}_{+}^{2}\mapsto \mathbb{R}$ defined by 
\begin{equation}
M_{p}\left( a,b\right) =M^{1/p}\left( a^{p},b^{p}\right) \text{ if }p\neq 0%
\text{ and }M_{0}\left( a,b\right) =a^{M_{x}\left( 1,1\right)
}b^{M_{y}\left( 1,1\right) },  \label{M_p}
\end{equation}%
where $M_{x}\left( x,y\right) $, $M_{y}\left( x,y\right) $ stand for the
first-order partial derivatives with respect to the first and second
component of $M(x,y)$,$\ $respectively.

The following theorem reveals that $M_{p}$ is also a mean, and it is called "%
$M$ mean of order $p$". Since the form of $M_{p}$ is similar to power mean $%
A_{p}$, it is also known simply as "power-type mean".

\begin{theorem}
\label{Theorem 2.1}Let $M$ be a differentiable mean on $\mathbb{R}_{+}^{2}$
and $M_{p}$ be defined by (\ref{M_p}). Then $M_{p}$ is also a mean. In
particular, $M_{0}=G$ if $M$ is symmetric.
\end{theorem}

\begin{proof}
We distinguish two cases to prove it.

Case 1: $p\neq 0$. Without loss of generality, we assume that $p>0$ and $%
b>a>0$. Since $M$ is a mean, we have $a^{p}<M\left( a^{p},b^{p}\right)
<b^{p} $, which implies that $a<M^{1/p}\left( a^{p},b^{p}\right) <b$, that
is, $M^{1/p}\left( a^{p},b^{p}\right) $ is also a mean.

Case 2: $p=0$. Clearly, it suffices to show that 
\begin{equation*}
M_{x}\left( 1,1\right) ,M_{y}\left( 1,1\right) \in \left( 0,1\right) \text{
and }M_{x}\left( 1,1\right) +M_{y}\left( 1,1\right) =1.\text{ }
\end{equation*}%
In fact, it is known that%
\begin{eqnarray*}
x &<&M\left( x+\vartriangle x,x\right) <x+\vartriangle x\text{ if }%
\vartriangle x>0, \\
x+ &\vartriangle &x<M\left( x+\vartriangle x,x\right) <x\text{ if }%
\vartriangle x<0,
\end{eqnarray*}%
which can be written as%
\begin{equation*}
0<\frac{M\left( x+\vartriangle x,x\right) -x}{\vartriangle x}<1.
\end{equation*}%
And so, 
\begin{equation*}
M_{x}\left( x,x\right) =\lim_{\vartriangle x\rightarrow 0}\frac{M\left(
x+\vartriangle x,x\right) -M\left( x,x\right) }{\vartriangle x}%
=\lim_{\vartriangle x\rightarrow 0}\frac{M\left( x+\vartriangle x,x\right) -x%
}{\vartriangle x}\in \left( 0,1\right) .
\end{equation*}%
Similarly, we can prove $M_{y}\left( x,x\right) \in \left( 0,1\right) $.

Differentiating for the identity $M\left( x,x\right) =x$ yields $M_{x}\left(
x,x\right) +M_{y}\left( x,x\right) =1$. Hence, $M_{0}$ is still a mean.

Particularly, if $M$ is symmetric, that is, $M\left( x,y\right) =M\left(
y,x\right) $, then it is easy to obtain that $M_{x}\left( x,y\right)
=M_{y}\left( y,x\right) $. From this it is deduced that $M_{x}\left(
x,x\right) =M_{y}\left( x,x\right) $, which together with $M_{x}\left(
x,x\right) +M_{y}\left( x,x\right) =1$ leads to $M_{x}\left( x,x\right)
=M_{y}\left( x,x\right) =1/2$. Thus, $M_{0}=G$.

This completes the proof.
\end{proof}

Applying the results in \cite{Yang.6(4).101.2005}, we can give a sufficient
condition for the monotonicity of $p$-order $M$ mean $M_{p}$.

\begin{lemma}
\label{Lemma 2.1}Let $M$ be a homogeneous and differentiable mean. Then the
function $M_{p}$ defined by (\ref{M_p}) is strictly increasing (decreasing)
in $p$ on $\mathbb{R}$ if $\mathcal{I}=\left( \ln M\right) _{xy}>(<)0$.
\end{lemma}

For $M=A,He,L,I,Y$, it has been proven that $\mathcal{I}=\left( \ln M\right)
_{xy}<(>)0$ in \cite{Yang.149286.2008}, and so all the corresponding $p$%
-order arithmetic mean (i.e., power mean) $A_{p}$, $p$-order Heroian mean $%
He_{p}$, $p$-order logarithmic mean $L_{p}$, $p$-order identric
(exponential) mean $I_{p}$ and $p$-order exponential-geometric mean $Y_{p}$
are strictly increasing in $p$ on $\mathbb{R}$. Now we shall show that the
first $p$-order Seiffert mean $P_{p}$, the second $p$-order Seiffert mean $%
T_{p}$, $p$-order Neuman-S\'{a}ndor mean $N_{p}$ and $p$-order
power-exponential mean $Z_{p}$ have the same monotonicity.

\begin{theorem}
\label{Theorem 2.2}The first $p$-order Seiffert mean $P_{p}$, the second $p$%
-order Seiffert mean $T_{p}$, $p$-order Neuman-S\'{a}ndor mean $N_{p}$ and $%
p $-order power-exponential mean $Z_{p}$ are strictly increasing in $p$ on $%
\mathbb{R}$.
\end{theorem}

\begin{proof}
By Theorem \ref{Lemma 2.1}, it suffices to show that $\mathcal{I}=\left( \ln
M\right) _{xy}<0$, where $M=P,T,N$.

(i) Direct computation yields%
\begin{eqnarray*}
\mathcal{I} &=&\left( \ln P\right) _{xy}=\frac{1}{\left( x-y\right) ^{2}}%
\allowbreak -\frac{1}{\arcsin ^{2}\frac{x-y}{x+y}}\tfrac{1}{\left(
x+y\right) ^{2}}-\frac{1}{2}\frac{1}{\arcsin \frac{x-y}{x+y}}\tfrac{x-y}{%
\sqrt{xy}\left( x+y\right) ^{2}} \\
&=&\frac{1}{\left( x-y\right) ^{2}}-\frac{4}{\left( x+y\right) ^{2}\left(
x-y\right) ^{2}}P^{2}\allowbreak -\frac{1}{\left( x+y\right) ^{2}\sqrt{xy}}P.
\end{eqnarray*}%
Using the known inequality $P>G=\sqrt{xy}$, we get $\allowbreak $%
\begin{equation*}
\mathcal{I}<\frac{1}{\left( x-y\right) ^{2}}-\frac{4}{\left( x+y\right)
^{2}\left( x-y\right) ^{2}}xy\allowbreak -\frac{1}{\left( x+y\right) ^{2}%
\sqrt{xy}}\sqrt{xy}=0.
\end{equation*}

(ii) In the same way, we have 
\begin{eqnarray*}
\mathcal{I} &=&\left( \ln T\right) _{xy}=\frac{1}{\left( x-y\right) ^{2}}-%
\frac{y}{\arctan ^{2}\frac{x-y}{x+y}}\allowbreak \frac{x}{\left(
x^{2}+y^{2}\right) ^{2}}-\frac{1}{\arctan \frac{x-y}{x+y}}\frac{x^{2}-y^{2}}{%
\left( x^{2}+y^{2}\right) ^{2}} \\
&=&\allowbreak \frac{1}{\left( x-y\right) ^{2}}-\frac{4xy}{\left(
x^{2}+y^{2}\right) ^{2}\left( x-y\right) ^{2}}T^{2}-\frac{2\left( x+y\right) 
}{\left( x^{2}+y^{2}\right) ^{2}}T\allowbreak .
\end{eqnarray*}%
The known inequality $T>A=\left( x+y\right) /2$ results in 
\begin{eqnarray*}
\mathcal{I} &<&\frac{1}{\left( x-y\right) ^{2}}-\frac{4xy}{\left(
x^{2}+y^{2}\right) ^{2}\left( x-y\right) ^{2}}\left( \frac{x+y}{2}\right)
^{2}-\frac{2\left( x+y\right) }{\left( x^{2}+y^{2}\right) ^{2}}\left( \frac{%
x+y}{2}\right) \\
&=&\allowbreak -\frac{xy}{\left( x^{2}+y^{2}\right) ^{2}}<0.
\end{eqnarray*}

(iii) We have 
\begin{eqnarray*}
\mathcal{I} &=&\left( \ln N\right) _{xy}=\frac{1}{\left( x-y\right) ^{2}}-%
\frac{1}{\func{arcsinh}^{2}\frac{x-y}{x+y}}\tfrac{2xy}{\left(
x^{2}+y^{2}\right) \left( x+y\right) ^{2}}-\tfrac{\sqrt{2}\left(
x^{2}+y^{2}+xy\right) }{\left( x+y\right) ^{2}\left( \sqrt{x^{2}+y^{2}}%
\right) ^{3}}\frac{x-y}{\func{arcsinh}\frac{x-y}{x+y}} \\
&=&\frac{1}{\left( x-y\right) ^{2}}-\tfrac{8xy}{\left( x-y\right) ^{2}\left(
x^{2}+y^{2}\right) \left( x+y\right) ^{2}}N^{2}-\tfrac{2\sqrt{2}\left(
x^{2}+y^{2}+xy\right) }{\left( x+y\right) ^{2}\left( \sqrt{x^{2}+y^{2}}%
\right) ^{3}}N.\allowbreak
\end{eqnarray*}%
Application of the inequality 
\begin{equation*}
N>\frac{A^{2}}{A_{2}}=\frac{\left( \frac{x+y}{2}\right) ^{2}}{\sqrt{\frac{%
x^{2}+y^{2}}{2}}}
\end{equation*}%
proved in \cite{Yang.arXiv:1208.0895} leads to 
\begin{equation*}
\mathcal{I}=\left( \ln N\right) _{xy}<\tfrac{1}{\left( x-y\right) ^{2}}-%
\tfrac{8xy}{\left( x-y\right) ^{2}\left( x^{2}+y^{2}\right) \left(
x+y\right) ^{2}}\left( \tfrac{\left( \frac{x+y}{2}\right) ^{2}}{\sqrt{\frac{%
x^{2}+y^{2}}{2}}}\right) ^{2}-\tfrac{2\sqrt{2}\left( x^{2}+y^{2}+xy\right) }{%
\left( x+y\right) ^{2}\left( \sqrt{x^{2}+y^{2}}\right) ^{3}}\tfrac{\left( 
\frac{x+y}{2}\right) ^{2}}{\sqrt{\frac{x^{2}+y^{2}}{2}}}=0.
\end{equation*}

(iv) Lastly, we prove the monotonicity of $Z_{p}$ in $p$. To this end, it
suffices to prove that the function 
\begin{equation*}
p\mapsto \ln Z_{p}=\frac{a^{p}}{a^{p}+b^{p}}\ln a+\frac{b^{p}}{a^{p}+b^{p}}%
\ln b
\end{equation*}%
is increasing on $\mathbb{R}$. In fact, we have 
\begin{equation*}
\frac{d}{dp}\left( \ln Z_{p}\right) =\allowbreak a^{p}b^{p}\frac{\left( \ln
a-\ln b\right) ^{2}}{\left( a^{p}+b^{p}\right) ^{2}}>0,
\end{equation*}%
which proves the monotonicity of $Z_{p}$ in $p$ and the whole proof is
completed.
\end{proof}

\section{Sharp inequalities among power-type means}

We first establish the relation between logarithmic mean of order $p$ and
the first Seiffert mean.

\begin{theorem}
\label{Theorem L_2<P}For $a,b>0$ with $a\neq b$, the inequality $L_{p}<P$ if
and only if $p\leq 2$.
\end{theorem}

\begin{proof}
Due to the symmetry, we assume that $a<b$. Then inequality $L_{p}\left(
a,b\right) <P\left( a,b\right) $ is equivalent with 
\begin{equation}
\left( \frac{x^{p}-1}{p\ln x}\right) ^{1/p}<\frac{x-1}{2\arcsin \frac{x-1}{%
x+1}},  \label{L_p<P}
\end{equation}%
where\ $x=a/b\in \left( 0,1\right) $.

Necessity. If $L_{p}<P$, then we have 
\begin{equation*}
\lim_{x\rightarrow 1}\frac{\left( \frac{x^{p}-1}{p\ln x}\right) ^{1/p}-\frac{%
x-1}{2\arcsin \frac{x-1}{x+1}}}{\left( x-1\right) ^{2}}=\allowbreak \frac{1}{%
24}p-\frac{1}{12}\leq 0,
\end{equation*}%
which indicates that $p\leq 2$.

Sufficiency. We prove the inequality (\ref{L_p<P}) holds if $p\leq 2$. By
Theorem \ref{Theorem 2.2}, it suffices to show that the inequality (\ref%
{L_p<P}) holds if $p=2$. Let the function $f_{1}$ be defined on $\left(
0,1\right) $ by 
\begin{equation*}
f_{1}\left( x\right) =\allowbreak 2\frac{\left( x+1\right) }{x-1}\arcsin ^{2}%
\frac{x-1}{x+1}-\ln x.
\end{equation*}%
Differentiation yields 
\begin{eqnarray*}
f_{1}^{\prime }\left( x\right) &=&\allowbreak -4\frac{\arcsin ^{2}\frac{x-1}{%
x+1}}{\left( x-1\right) ^{2}}+4\left( \arcsin \frac{x-1}{x+1}\right) \frac{1%
}{\sqrt{x}\left( x-1\right) }-\frac{1}{x}\allowbreak \\
&=&-\frac{\left( x-1-2\sqrt{x}\arcsin \frac{x-1}{x+1}\right) ^{2}}{x\left(
x-1\right) ^{2}}<0,
\end{eqnarray*}%
which shows that $f_{1}$ is decreasing on $\left( 0,1\right) $. Hence $%
f_{1}\left( x\right) >\lim_{x\rightarrow 1^{-}}f_{1}\left( x\right) =0$, and
so 
\begin{equation*}
\frac{x^{2}-1}{2\ln x}<\frac{\left( x-1\right) ^{2}}{4\left( \arcsin \frac{%
x-1}{x+1}\right) ^{2}},
\end{equation*}%
this proves the sufficiency and the proof is complete.
\end{proof}

Secondly, we show the relation between the first Seiffert mean of order $p$
and Neuman-S\'{a}ndor mean.

\begin{theorem}
\label{Theorem P_2<N}For $a,b>0$ with $a\neq b$, the inequality $P_{p}<N$ if
and only if $p\leq 2$.
\end{theorem}

\begin{proof}
Similarly, we assume that $a<b$. Then inequality $P_{p}\left( a,b\right)
<N\left( a,b\right) $ is equivalent with 
\begin{equation}
\left( \frac{x^{p}-1}{2\arcsin \frac{x^{p}-1}{x^{p}+1}}\right) ^{1/p}<\frac{%
x-1}{2\ln \frac{x-1+\sqrt{2\left( x^{2}+1\right) }}{x+1}},  \label{P_p<N}
\end{equation}%
where\ $x=a/b\in \left( 0,1\right) $.

Necessity. If $P_{p}<N$, then we have 
\begin{equation*}
\lim_{x\rightarrow 1}\frac{\left( \frac{x^{p}-1}{2\arcsin \frac{x^{p}-1}{%
x^{p}+1}}\right) ^{1/p}-\frac{x-1}{2\ln \frac{x-1+\sqrt{2\left(
x^{2}+1\right) }}{x+1}}}{\left( x-1\right) ^{2}}=\allowbreak \frac{1}{12}p-%
\frac{1}{6}\leq 0,
\end{equation*}%
which implies that $p\leq 2$.

Sufficiency. We prove the inequality (\ref{P_p<N}) holds if $p\leq 2$. By
Theorem \ref{Theorem 2.2}, it suffices to show that the inequality (\ref%
{P_p<N}) holds if $p=2$. We define the function $f_{2}$ by 
\begin{equation*}
f_{2}\left( x\right) =\allowbreak 2\frac{\ln ^{2}\frac{x-1+\sqrt{2\left(
x^{2}+1\right) }}{x+1}}{x-1}\left( x+1\right) -\arcsin \frac{x^{2}-1}{x^{2}+1%
},\text{ \ }x\in \left( 0,1\right) .
\end{equation*}%
Differentiation yields 
\begin{eqnarray*}
f_{2}^{\prime }\left( x\right) &=&-4\frac{\ln ^{2}\frac{x-1+\sqrt{2\left(
x^{2}+1\right) }}{x+1}}{\left( x-1\right) ^{2}}+4\frac{\ln \frac{x-1+\sqrt{%
2\left( x^{2}+1\right) }}{x+1}}{x-1}\allowbreak \frac{\sqrt{2}}{\sqrt{x^{2}+1%
}}-\allowbreak \frac{2}{x^{2}+1}\allowbreak \\
&=&-\left( 2\frac{\ln \frac{x-1+\sqrt{2\left( x^{2}+1\right) }}{x+1}}{x-1}-%
\frac{\sqrt{2}}{\sqrt{x^{2}+1}}\right) ^{2}<0,
\end{eqnarray*}%
which implies that $f_{2}$ is decreasing on $\left( 0,1\right) $. Therefore $%
f_{2}\left( x\right) >\lim_{x\rightarrow 1^{-}}f_{2}\left( x\right) =0$, and
then 
\begin{equation*}
\frac{x^{2}-1}{2\arcsin \frac{x^{2}-1}{x^{2}+1}}<\left( \frac{x-1}{2\ln 
\frac{x-1+\sqrt{2\left( x^{2}+1\right) }}{x+1}}\right) ^{2},
\end{equation*}%
this proves the sufficiency and the proof is finished.
\end{proof}

Thirdly, let us prove the inequality for Neuman-Sandor mean and Heronian
mean of order $p$.

\begin{theorem}
\label{Theorem N<He_2}For $a,b>0$ with $a\neq b$, the inequality $N<He_{p}$
if and only if $p\geq 2$.
\end{theorem}

\begin{proof}
We assume that $a<b$. Then inequality $N\left( a,b\right) <He_{p}\left(
a,b\right) $ is equivalent with 
\begin{equation}
\frac{x-1}{2\ln \frac{x-1+\sqrt{2\left( x^{2}+1\right) }}{x+1}}<\left( \frac{%
x^{p}+x^{p/2}+1}{3}\right) ^{1/p},  \label{N<He_p}
\end{equation}%
where\ $x=a/b\in \left( 0,1\right) $.

Necessity. If $N<He_{p}$, then we have 
\begin{equation*}
\lim_{x\rightarrow 1}\frac{\frac{x-1}{2\ln \frac{x-1+\sqrt{2\left(
x^{2}+1\right) }}{x+1}}-\left( \frac{x^{p}+x^{p/2}+1}{3}\right) ^{1/p}}{%
\left( x-1\right) ^{2}}=\allowbreak \allowbreak \frac{1}{6}-\frac{1}{12}%
p\leq 0,
\end{equation*}%
which reveals that $p\geq 2$.

Sufficiency. We prove the inequality (\ref{N<He_p}) holds if $p\geq 2$. By
Theorem \ref{Theorem 2.2}, it suffices to show that the inequality (\ref%
{N<He_p}) holds if $p=2$. To this end, we define the function $f_{3}$ by 
\begin{equation*}
f_{3}\left( x\right) =\allowbreak \frac{x-1}{\sqrt{\frac{x^{2}+x+1}{3}}}%
-2\ln \frac{x-1+\sqrt{2\left( x^{2}+1\right) }}{x+1}.
\end{equation*}%
Differentiation yields 
\begin{eqnarray*}
f_{3}^{\prime }\left( x\right) &=&\allowbreak \frac{1}{\sqrt{\frac{x^{2}+x+1%
}{3}}}-\frac{2x+1}{6}\frac{x-1}{\left( \frac{x^{2}+x+1}{3}\right) ^{\frac{3}{%
2}}}-\frac{2\sqrt{2}}{\left( x+1\right) \sqrt{x^{2}+1}}\allowbreak \\
&=&\sqrt{2}\frac{x\sqrt{\frac{x^{2}+1}{2}}-2\left( \frac{x^{2}+x+1}{3}%
\right) ^{3/2}+\left( \frac{x^{2}+1}{2}\right) ^{3/2}}{\left( x+1\right) 
\sqrt{x^{2}+1}\left( \sqrt{\frac{x^{2}+x+1}{3}}\right) ^{3}} \\
&=&-\sqrt{2}\tfrac{\left( \sqrt{\frac{x^{2}+1}{2}}-\sqrt{\frac{1}{3}x+\frac{1%
}{3}x^{2}+\frac{1}{3}}\right) ^{2}\left( \sqrt{\frac{x^{2}+1}{2}}+2\sqrt{%
\frac{1}{3}x+\frac{1}{3}x^{2}+\frac{1}{3}}\right) }{\left( x+1\right) \sqrt{%
x^{2}+1}\left( \sqrt{\frac{x^{2}+x+1}{3}}\right) ^{3}}<0,
\end{eqnarray*}%
which shows that $f_{3}$ is decreasing on $\left( 0,1\right) $. Hence $%
f_{3}\left( x\right) >\lim_{x\rightarrow 1^{-}}f_{3}\left( x\right) =0$, and
so 
\begin{equation*}
\frac{x-1}{\sqrt{\frac{x^{2}+x+1}{3}}}>2\ln \frac{x-1+\sqrt{2\left(
x^{2}+1\right) }}{x+1},
\end{equation*}%
which implies that the inequality (\ref{N<He_p}) holds if $p=2$, that is,
the sufficiency holds.

Thus the proof ends.
\end{proof}

Next we further prove the sharp inequality for identric (exponential) mean
and power-exponential mean of order $p$.

\begin{theorem}
\label{Theorem I<Z_1/3}For $a,b>0$ with $a\neq b$, the inequality $I<Z_{p}$
if and only if $p\geq 1/3$.
\end{theorem}

\begin{proof}
We assume that $a<b$. Then inequality $I\left( a,b\right) <Z_{p}\left(
a,b\right) $ is equivalent with 
\begin{equation}
e^{-1}x^{x/\left( x-1\right) }<x^{\frac{x^{p}}{x^{p}+1}},  \label{I-Z_p}
\end{equation}%
where\ $x=a/b\in \left( 0,1\right) $.

Necessity. If $I\left( a,b\right) <Z_{p}\left( a,b\right) $ is true, then we
have 
\begin{equation*}
\lim_{x\rightarrow 1}\frac{e^{-1}x^{x/\left( x-1\right) }-x^{\frac{x^{p}}{%
x^{p}+1}}}{\left( x-1\right) ^{2}}=\frac{1}{12}-\frac{1}{4}p\leq 0,
\end{equation*}%
which yields $p\geq 1/3$.

Sufficiency. It has been proved in \cite[(5.7)]{Yang.10(3).2007} that $%
I<Z_{1/3}$. By the monotonicity proved in Theorem \ref{Theorem 2.2}, it is
derived that $I<Z_{1/3}\leq Z_{p}$ if $p\geq 1/3$.

This completes the proof.
\end{proof}

Lastly, we will show that the inequality $Z_{2/3}<Y$ is the best.

\begin{theorem}
\label{Theorem Z_2/3<Y}For $a,b>0$ with $a\neq b$, the inequality $Z_{p}<Y$
if and only if $p\leq 2/3$.
\end{theorem}

\begin{proof}
We assume that $a<b$. Then inequality $Z_{p}\left( a,b\right) <Y\left(
a,b\right) $ is equivalent with 
\begin{equation}
x^{\frac{x^{p}}{x^{p}+1}}<e^{-1}x^{x/\left( x-1\right) }\exp \left( 1-\frac{%
x\ln ^{2}x}{\left( x-1\right) ^{2}}\right) ,  \label{Y-Z_p}
\end{equation}%
where\ $x=a/b\in \left( 0,1\right) $.

Necessity. If $Z_{p}\left( a,b\right) <Y\left( a,b\right) $ is valid, then
we have 
\begin{equation*}
\lim_{x\rightarrow 1}\tfrac{\ln Z_{p}\left( x,1\right) -\ln Y\left(
x,1\right) }{\left( x-1\right) ^{2}}=\lim_{x\rightarrow 1}\tfrac{\frac{x^{p}%
}{x^{p}+1}\ln x-\frac{x}{x-1}\ln x+1-\left( 1-\frac{x\ln ^{2}x}{\left(
x-1\right) ^{2}}\right) }{\left( x-1\right) ^{2}}=\frac{1}{4}p-\frac{1}{6}%
\leq 0,
\end{equation*}%
which leads to $p\leq 2/3$.

Sufficiency. It has been proved in \cite[(5.12)]{Yang.10(3).2007} that $%
Z_{2/3}<Y$. By the monotonicity of $Z_{p}$ in $p$, it is derived that $%
Y>Z_{2/3}\geq Z_{p}$ if $p\leq 2/3$.

This completes the proof.
\end{proof}

\section{Remarks and a conjecture}

\begin{remark}
\label{Remark 4.1}With $a^{p}\rightarrow a$, $b^{p}\rightarrow b$, the Lemma %
\ref{Theorem P_2<N} can be restated as:

For $a,b>0$ with $a\neq b$, the inequalitiesit $P<N_{p}$ holds if and only
if $p\geq 1/2$.
\end{remark}

\begin{remark}
\label{Remark 4.2}In the same way, the Lemma \ref{Theorem N<He_2} is
equivalent with:

For $a,b>0$ with $a\neq b$, the inequalitiesit $N_{p}<He$ holds if and only
if $p\leq 1/2$.
\end{remark}

\begin{remark}
\label{Remark 4.3}From Theorem \ref{Theorem L_2<P}, Remark \ref{Remark 4.1}
and \ref{Remark 4.2} the chain of inequalities for means (\ref{Yang1}) can
be refined as 
\begin{equation}
L_{2}<P<N_{1/2}<He<A_{2/3}<I<Z_{1/3}<Y_{1/2}.  \label{Yang3}
\end{equation}%
Also, it is clear that all the constants located in lower right corner are
the best.
\end{remark}

The chain of inequalities for means (\ref{Yang3}) is very nice.
Unfortunately, it is not contain the second power-type Seiffert mean $T_{p}$%
. From (\ref{T_P-He}) and (\ref{Yang1}) it is easy to obtained that 
\begin{equation}
T_{2/5}<He<A_{2/3}<I<Z_{1/3}<Y_{1/2}.  \label{Chu-Yang}
\end{equation}%
If $N_{1/2}<$ $T_{2/5}$ holds, then we can get a more nice chain of
inequalities for power-type means 
\begin{equation}
L_{2}<P<N_{1/2}<T_{2/5}<He<A_{2/3}<I<Z_{1/3}<Y_{1/2}.  \label{Yang4}
\end{equation}%
Computation by mathematical software yields that if $N<T_{4/5}$ then 
\begin{equation*}
\lim_{x\rightarrow 1}\frac{N\left( x,1\right) -T_{p}\left( x,1\right) }{%
\left( x-1\right) ^{2}}=\frac{1}{6}-\frac{5}{24}p\leq 0,
\end{equation*}%
which implies that $p\geq 4/5$. And, we have 
\begin{equation*}
N\left( 0^{+},1\right) -T_{4/5}\left( 0^{+},1\right) =\frac{1}{2\ln \left( 
\sqrt{2}+1\right) }-\left( \frac{2}{\pi }\right) ^{5/4}<0.
\end{equation*}%
Hence we propose the following conjecture:

\begin{conjecture}
For $a,b>0$ with $a\neq b$, the inequalities $N<T_{p}$ holds if and only if $%
p\geq 4/5$.
\end{conjecture}

\end{document}